\documentclass[11pt]{article}
\usepackage[utf8]{inputenc}
\usepackage[usenames,dvipsnames,svgnames,table]{xcolor}	
\usepackage[]{amsmath}
\usepackage{cases}
\usepackage{mathrsfs}
\usepackage{empheq}
\usepackage[colorlinks=true]{hyperref}
\hypersetup{colorlinks,
citecolor=brown, linkcolor=blue, linktocpage
}
\usepackage{amssymb}
\usepackage{bbm}
\usepackage{enumerate}   
\usepackage{geometry}
\usepackage{amsfonts}
\usepackage{color}
\usepackage[]{amsthm} 
 \usepackage[shortlabels]{enumitem}
 \usepackage{cleveref}

\newtheorem{theorem}{Theorem}[section]
\newtheorem{lemma}[theorem]{Lemma}

\newtheorem{proposition}[theorem]{Proposition}
\newtheorem{definition}[theorem]{Definition}

\newtheorem{remark}[theorem]{Remark}
\newtheorem{assumption}[theorem]{Assumption}

\def\E{\mathbb{E}}
\def\R{\mathbb{R}}
\def\Ps{\mathbb{P}}
\def\Q{\mathbb{Q}}
\def\N{\mathbb{N}}

\def\W{\mathbb{W}}
\def\1{\mathbbm{1}}

\def\D1{\frac{\partial}{\partial x}}

\newcommand{\Ff}{\mathcal F}

\newcommand{\Pp}{\mathcal P}

\numberwithin{equation}{section}

\usepackage{xspace}

\title{Propagation of chaos for stochastic particle systems with singular mean-field interaction of $L^q-L^p$ type}       
\author{Milica Toma\v sevi\'c\footnote{CMAP, CNRS, École polytechnique, Institut Polytechnique de
Paris, 91120 Palaiseau, France. milica.tomasevic@polytechnique.edu}}
\date{}                       

\begin{document}

\maketitle                    

\begin{abstract}
In this work, we prove the well-posedness and propagation of chaos for a stochastic particle system in mean-field interaction under the assumption that the interacting kernel belongs to a suitable $L_t^q-L_x^p$ space. Contrary to the large deviation principle approach recently proposed in~\cite{Mauri}, the main ingredient of the proof here are the \textit{Partial Girsanov transformations} introduced in~\cite{JTT} and developed in a general setting in this work.
\end{abstract}



\section{Introduction }
The goal of this paper is to study the well-posedness and the convergence as $N\to \infty$  of the following particle system taking values in $\R^d$
\begin{equation}
\label{PS0}
\begin{cases}
&dX_t^{i,N}= \frac{1}{N} \sum_{j=1}^N 
b(t,X_t^{i,N},X_t^{j,N})~dt +  dW_t^i, \quad t>0, i\leq N\\
& X_0^{i,N} \text{ i.i.d. and  independent of } W:=(W^i, 1\leq
i \leq N),
\end{cases}
\end{equation}
towards the non-linear stochastic differential equation 
\begin{equation}
\label{SDE}
\begin{cases}
& dX_t= \int b(t,X_t,y) \rho_t(y) \ dy \ dt
+   dW_t, \quad t> 0, \\
& \rho_t(y)dy:=  \mathcal{L}( X_{t}),\quad X_0 \sim \rho_0(x)dx,
\end{cases}
\end{equation}
under the following assumption 
\begin{assumption}
\label{H1}
For  $x,y\in \R^d$ and $t>0$, one has $|b(t,x,y)|\leq h_t(x-y)$ for some $ h \in L^q_{loc}(\R_+;L^p(\R^d)) $, where $p,q\in(2,\infty)$ satisfy $\frac{d}{p}+\frac{2}{q}<1$.
\end{assumption}
In \eqref{PS0} and \eqref{SDE} $W^i$'s  and $W$
 are 
 independent standard $d$-dimensional Brownian motions defined on a probability space $(\Omega, \mathcal{F}, \mathbb{P})$. 
 
 Assumption~\ref{H1} is the so-called $L_t^q-L_x^p$ assumption on the drift and it goes back to Krylov-R\" ockner~\cite{KryRoc-05} in the framework of linear and non-interacting SDEs. When it comes to non-linear stochastic processes \eqref{SDE}, it has been studied in  R\"ockner and Zhang~\cite{RocknerZhang}. There, the authors conjecture 
 the propagation of chaos of \eqref{PS0} towards \eqref{SDE} under an $L_t^q-L_x^p$ type assumption on the interaction kernel. This has been  very recently  established in the literature as a consequence of the large deviation principle for such particle system proved in Hoeksama \textit{et al.}~\cite{Mauri}.
 
  In this work, we propose a completely different approach 
relying on the new techniques introduced by Jabir \textit{et al.}~\cite{JTT} to prove the propagation of chaos for  the particle system with both non-Markovian and singular interaction related to the parabolic-parabolic one-dimensional  Keller-Segel model. The main difference w.r.t.~\cite{Mauri} is that we get the propagation of chaos directly working on the martingale problem on the level of the particle system. In addition, we extend a technique developed in a very specific case to the general setting of~\eqref{PS0}.

The plan of the paper is the following. In Section~\ref{sec:main} we state the main results of this work. Then, weak well-posedness is established for \eqref{PS0} inspired by the approach in~\cite{KryRoc-05} in this setting (Section \ref{sec:existence}). Namely, by means of Girsanov transform, the interaction is added to a driftless system of $N$ independent Brownian motions. The latter is justified by controlling the exponential martingale arising from the transform through a Novikov condition.\\
Then in Section~\ref{sec:propchaos}, to prove the propagation of chaos property, we first obtain the tightness w.r.t. $N\geq 1$ of the laws $\Pi^N$ of $\mu^N$. Then, we show that all the limit points of $\Pi^N$ are concentrated around the law $\Q$ of the weak solution to \eqref{SDE}. 
Informally speaking, we show the convergence of the martingale problem on the level of the particle system towards the non-linear martingale problem at the limit. This is a well established tactic (see e.g.~\cite{Sznitman}) and it is successfully applied in case of discontinuous interaction in~\cite{BossyTalay}. In this framework due to the singular nature of the drift, this is the most delicate part of the proof for which we need to go beyond the approach of~\cite{KryRoc-05}.

Namely, the Girsanov transform used to define the particle system should be the main tool to prove the above tightness and convergence properties. However, the control on the exponential martingale related to this \textit{full} Girsanov transform explodes with $N\to \infty$ and as such is of no use. As noticed in~\cite[p. 180]{Sznitman}, what one needs to control when passing to the limit are functionals of finite number of particles (usually at most 4) as the particle system is exchangeable. Hence, removing the drift of all the particles when exhibiting the above controls is unnecessarily expensive. To circumvent this inconvenience, following~\cite{JTT} we introduce \textit{Partial Girsanov Transforms} that remove drifts of a finite number $k<N$ of particles and their dependence of the rest $N-k$ particles. It turns out that such transforms are especially suitable for our multi-dimensional setting and lead to non-divergent estimates on the corresponding exponential martingales. 

To finish this part, let us discuss here an alternative hypothesis to Assumption~\ref{H1} on which our technique also applies. In fact, we can ask less than  $L^p$ condition on the whole space for the dominating function $h$. For instance, in~\cite{RocknerZhang} the authors work with localised $L^q-L^p$ spaces and a typical example of a singular interaction treated is $b_t(x,y)= \frac{a_t(x,y) }{|x-y|^\alpha}$ for $|a_t(x,y)|\leq \kappa |x-y|$ where $\alpha \in [1,2)$ and $\kappa >0$. In the framework of this paper, such kernels are also admissible even if they do not satisfy Assumption~\ref{H1}. In fact, it is enough to suppose the following alternative.
\begin{assumption}
\label{H2}
For  $x,y\in \R^d$ and $t>0$, we suppose that $|b(t,x,y)|\leq h_t(x-y)$ for some $ h \in L^q_{loc}(\R_+;L_{loc}^p(\R^d)) $, where $p,q\in(2,\infty)$ satisfy $\frac{d}{p}+\frac{2}{q}<1$ and the function $H(T):=\int_0^T \sup_{|x|>1}|h_t(x)|^2 dt $ is an increasing function from $\R_+$ to $\R_+$.
\end{assumption}
In this case, the key lemma (Lemma~\ref{MainLemmas1}) in our computation is easily adapted and from there on it is straightforward to adapt the rest of the proofs (see the end of Section~\ref{sec:existence}).

\subsection{Main Results}
\label{sec:main}
Let $t>0$. As the interaction kernel $b(t,\cdot,\cdot)$ has only integrability properties w.r.t. the second or third variable, it may be unbounded or not well defined in certain points. Let us denote by $\mathcal{N}_b(t)$ the following set:
\begin{multline*}
    \mathcal{N}_b(t)=\Big\{(x,y)\in \R^d\times \R^d:\lim_{(x',y')\to (x,y)} |b(t,x',y')|=\infty \\
    \text{  or } \lim_{(x',y')\to (x,y)} |b(t,x',y')| \text{ does not exist} \Big\}.
\end{multline*}
However, as $|b(t,x,y)|\leq h_t(x-y)$ and   $h_t\in L^p(\R^d)$, the set  $\mathcal{N}_b(t)$ is of Lebesgue's measure zero in $\R^d\times \R^d$.
Then, in order to ensure that the drift term makes sense, for $N\geq 1$, the particle system reads 
\begin{equation}
\label{PS}
\begin{cases}
&dX_t^{i,N}= \frac{1}{N} \sum_{j=1,j\neq i}^N 
b(t,X_t^{i,N},X_t^{j,N}) \mathbb{1}_{{\{(X_t^{i,N},X_t^{j,N})\notin \mathcal{N}_b(t)\}}}  ~dt + \sqrt{2} dW_t^i, \\
& X_0^{i,N} \text{ i.i.d. and  independent of } W:=(W^i, 1\leq
i \leq N).
\end{cases}
\end{equation}
 Notice here that we assume that a particle does not interact with itself and that we set an interaction to zero every time the set $\mathcal{N}_b(t)$ is visited. As we will construct the particle system by means of a Girsanov transformation, the law of $(X^N)$ will be absolutely continuous w. r. t. Wiener measure. 
 Lebesgue measure of the set $\{t>0, (X_t^{i,N},X_t^{j,N})\in \mathcal{N}_b(t)\}$ will thus be zero. Hence, the dynamics \eqref{PS0} and \eqref{PS} are essentially the same. 

In practical examples, $h_t$ explodes (or it is not well defined) only in one point (zero) and thus, $b(t,\cdot,\cdot)$ may explode (or not be well defined) on the line $x=y$. In order to keep the result as general as possible, we will rather work with the sets $(\mathcal{N}_b(t))_{t\geq 0}$. However, the reader should have in mind that this is just a technicality and that, in practice $ \mathbb{1}_{{\{(X_t^{i,N},X_t^{j,N})\notin \mathcal{N}_b(t)\}}} $ becomes $\mathbb{1}_{{\{X_t^{i,N}\neq
X_t^{j,N}\}}} $ (as in~\cite{JTT}).

Our first main result is the weak well-posedness of \eqref{PS}.  
\begin{theorem}
\label{PSexistence}
Let Assumption~\ref{H1} hold.  Given $0<T<\infty$ and $N \in \mathbb{N}$,
 there exists a weak solution $(\Omega,\Ff,
(\Ff_t;\,0\leq t\leq T),\Q^N, W, X^N)$ to the $N$-interacting particle
system \eqref{PS} that satisfies, for any $1\leq i \leq N$,
\begin{equation}\label{DriftL2}
 \Q^N\left(\int_0^T\left(\frac{1}{N}\sum_{j=1,j\neq i}^N b(t,X_t^{i,N},X_t^{j,N}) \mathbb{1}_{{\{(X_t^{i,N},X_t^{j,N})\notin \mathcal{N}_b(t)\}}}\right)^2\,dt<\infty\right)=1.
\end{equation}
\end{theorem}
In view of Karatzas and Shreve~\cite[Chapter 5, Proposition 3.10]{KaratzasShreve}, one has that weak uniqueness holds in the class of weak solutions satisfying \eqref{DriftL2}.

Before we state the propagation of chaos result, we formulate the  martingale problem associated to \eqref{SDE}. It is classical that a suitable notion of weak solution to \eqref{SDE}
is equivalent to the notion of solution to (MP) (see e.g.~\cite{KaratzasShreve}).

\begin{definition}
$\Q \in \mathcal{P}(C[0,T];\R^d)$ is a solution to (MP) if:
\begin{enumerate}[(i)]
\item $\Q_0 = \mu_0$;
\item For any $t\in (0,T]$ and any $r> 1$, the one dimensional time marginal~$\Q_t$
of $\Q$ has ~a density~ $\rho_t$ w.r.t. Lebesgue measure on $\R^d$ which belongs to $L^r(\R^d)$ ~and~satisfies~
$$ \exists C_T,~~\forall\, 0<t\leq T, ~~~ \Vert \rho_t\Vert_{L^r(\R^d)}
\leq \frac{C_{T}}{t^{\frac{d}{2}(1-\frac{1}{r})}}; $$
\item Denoting by $(x(t);\,t\leq T)$ the canonical process of
$C([0,T];\R^d)$, we have: For any $f \in C_b^2(\R^d)$, the process defined
by
$$M_t:=f(x(t))-f(x(0))-\int_0^t\left( \nabla f(x(s))\cdot \Big( \int b(s,x(s),y) \rho_s(y)dy \Big)+ \triangle f(x(s))\right)\,ds$$
is a $\Q$-martingale w.r.t. the canonical filtration.
\end{enumerate}
\end{definition}
\begin{remark}
Under Assumption~\ref{H1} and supposing that the measure $\mu_0$  has a finite $\beta$-order moment for some $\beta>2$, the martingale problem (MP) admits a unique solution according to~\cite[Thm. 1.1]{RocknerZhang}. There the authors prove the Gaussian estimates punctually  on the marginal densities. In the martingale formulation it is enough to impose such estimates in $L^r$-norms as, along with Assumption~\ref{H1}, this will ensure that all the terms in the definition of the process $(M)$ are well defined.
\end{remark}

We are ready to state our second main result, the propagation of chaos of \eqref{PS}. In practice one deals with interaction kernels in the form of convolutions, that are well defined and continuos almost everywhere (like $\pm \frac{x}{|x|^r}$). Hence,  it is reasonable to assume that $b(t,\cdot,\cdot)$ is continuous outside of $\mathcal{N}_b(t)$.
\begin{theorem}
\label{th:PSconvergence}
In addition to Assumption~\ref{H1}, assume that for any $t>0$, $b(t,\cdot,\cdot)$ is continuous outside of the set $\mathcal{N}_b(t)$.
Assume that the $X_0^{i,N}$'s are i.i.d. and~ that the initial distribution of $X_0^{1,N}$
is the measure $\mu_0$ that for some $\beta>2$ has finite $\beta$-order moment . 
Then, the empirical measure
$\mu^N=\frac{1}{N}\sum_{i=1}^N\delta_{X^{i,N}}$ of \eqref{PS}
converges in the weak sense, when $N \to \infty$,  to the
unique weak solution of \eqref{SDE}.
\end{theorem}
\begin{remark}
 \label{rem:altH}
 Replacing everywhere in this section Assumption~\ref{H1}  by Assumption~\ref{H2}, all the results still hold.
 \end{remark}
 
 \section{Proof of Theorem \ref{PSexistence}}
\label{sec:existence}
 We start from a probability space $(\Omega,\Ff,(\Ff_t;\,0\leq t\leq
T),\W)$ on which $d$ - dimensional Brownian motions 
$(W^1, \dots, W^N)$ and the random variables $X_0^{i,N}$(see \eqref{PS}) are defined. Set $\bar{X}_t^{i,N} := X_0^{i,N} + W_t^i
~(t\leq T)$ and $\bar{X}:=(\bar{X}^{i,N},1\leq i\leq N)$.
Denote the drift terms in \eqref{PS} by  $b^{i,N}_t(x)$, $x\in C([0,T];\R^d)^N$,
and the vector of all the drifts as $B_t^N(x)=(b^{1,N}_t(x), \dots, b^{N,N}_t(x))$.
For a fixed $N \in \N$, consider
$$Z_T^N:= \exp \left\{ \int_0^T B_t^N(\bar{X}) \cdot dW_t - \frac{1}{2}\int_0^T \left| B_t^N(\bar{X}) \right|^2dt \right\}.$$
To prove Theorem \ref{PSexistence}, it suffices to prove the following Novikov condition holds true
(see e.g.~\cite[Chapter 3, Proposition 5.13]{KaratzasShreve}):
For any $T>0$, $N \geq 1$, $\kappa >0$, there exists $C(T,N, \kappa)$ such that
\begin{equation} \label{ineq:exp-G-N}
\E_\W \left( \exp\left\{ \kappa \int_0^T |B_t^N(\bar{X})|^2dt\right\} \right)\leq
C(T,N, \kappa).
\end{equation}
Drop the index $N$ for simplicity. Using the definition of
$(B^N_t)$ and Jensen's inequality one has

  \begin{align*}
  &\E_{\W}\left[\exp\left\{\kappa\int_{0}^{T}\left|B_t^N(\bar{X})\right|^2\,dt\right\}\right]\leq\E_{\W}\left[\exp\left\{\frac{1}{N}\sum_{i=1}^{N}\frac{1}{N}\sum_{j=1,j\neq i}^N \int_{0}^{T}\kappa N |b(t,\bar{X_t}^i,\bar{X_t}^j)|^2\,dt\right\}\right],
\end{align*}
from which we deduce

$$   \E_{\W}\left[\exp\left\{\kappa\int_{0}^{T}\left|B_t^N(\bar{X})\right|^2\,dt\right\}\right]\leq \frac{1}{N}\sum_{i=1}^{N}\frac{1}{N}\sum_{j=1,j\neq i}^{N}
\E_{\W}\left[\exp\left\{\kappa
N\int_{0}^{T}|b(t,\bar{X_t}^i,\bar{X_t}^j)|^2\,dt\right\}\right].$$
Assume for a moment that for $i,j\leq N$ such that $j\neq i$ one has
\begin{equation}
\label{eq:forNovikov}
\E_{\W}\left[\exp\left\{\kappa
N\int_{0}^{T}|b(t,\bar{X_t}^i,\bar{X_t}^j)|^2\,dt\right\}\right]\leq C(T,N).
\end{equation}
Then \eqref{ineq:exp-G-N} is satisfied and the proof is finished. 

The rest of the proof will be devoted to establishing \eqref{eq:forNovikov}. Actually, we will prove the following more general statement: 
\begin{proposition}
\label{est1}
 Let $T>0$ and let Assumption~\ref{H1} hold. Let $w:=(w_t)$ be a $(\mathcal{G}_t)$-Brownian motion with an arbitrary
initial distribution $\mu_0$ on some probability space equipped with a
probability measure $\mathbb{P}$ and a filtration~$(\mathcal{G}_t)$. 
Suppose that the
filtered probability space is rich enough to support a continuous process $Y$ independent
 of~$(w_t)$.~  For any
 $\alpha>0$, one has 
\begin{equation*}
\E_{\Ps}\left[ \exp\left\{ \alpha \int_0^T |b(t,w_t,Y_t)|^2 dt\right\}\right]
\leq C(T, \alpha),
\end{equation*}
where $C(T, \alpha)$ depends only on $T$ and $\alpha$, but
does neither depend on the law~$\mathcal{L}(Y)$ nor of $\mu_0$.
\end{proposition}
To prove Proposition \ref{est1}, we need some preparation in form of two auxiliary Lemmas. 
First, for $(t,x)\in [0,T]\times \R^d$, denote by  $g_t(x):=\frac{1}{(2 \pi t)^{d/2}}e^{-\frac{|x|^2}{2t}}$.
Note that for any $t>0$ and any $p\geq 1$, one has
 \begin{equation}
 \label{eq:gaussianLp}
\|g_t\|_p=\dfrac{C_p}{t^{\frac{d}{2}(1-\frac{1}{p})}}.
 \end{equation}
 From now on, when we work on  a probability space equipped with a probability measure $\Ps$ and a filtration $(\mathcal{G}_t)_{t\geq 0}$, for any $t>0$, we will denote by $\E^{\mathcal{G}_{t}}_\Ps$ the conditional expectation w.r.t. $\mathcal{G}_{t}$, i.e. $\E^{\mathcal{G}_{t}}_\Ps(\cdot):=\E_\Ps(\cdot | \mathcal{G}_{t})$.
 \begin{lemma}
\label{MainLemmas1}
Suppose Assumption~\ref{H1} holds. 
Let $w:=(w_t)$ be a $(\mathcal{G}_t)$-Brownian motion with an arbitrary
initial distribution $\mu_0$ on some probability space equipped with a
probability measure $\mathbb{P}$ and a filtration~$(\mathcal{G}_t)$.
There exists a universal real number $C_0>0$ such that
$$ \forall x\in C([0,T];\R),~~\forall 0\leq t_1\leq t_2\leq T,~~
  \int_{t_1}^{t_2}\E^{\mathcal{G}_{t_1}}_\Ps |b(t,w_t,x_t)|^2dt
\leq C_0(T)
 (t_2-t_1)^{\frac{q-2}{q} - \frac{d}{p}} .$$
\end{lemma}
\begin{proof}
Using Assumption~\ref{H1}
one has
$$I:=\int_{t_1}^{t_2}\E^{\mathcal{G}_{t_1}}_\Ps |b(t,w_t,x_t)|^2dt \leq \int_{t_1}^{t_2}\int h^2_t(y+w_{t_1}-x_t)g_{t-t_1}(y)dy \ dt$$

Applying H\"older inequality in space with $\frac{p}{2}>1$ and afterwards in time with $\frac{q}{2}>1$, one has
$$I\leq \int_{t_1}^{t_2} \|h_t\|^2_{L^p(\R^d)}\|g_{t-t_1}\|_{L^{\frac{p}{p-2}}(\R^d)}\leq  \|h\|_{L^q((0,T);L^p(\R^d))} \left(\int_{t_1}^{t_2}\left(\|g_{t-t_1}\|_{L^{\frac{p}{p-2}}(\R^d)}\right)^{\frac{q}{q-2}} dt \right)^{\frac{q-2}{q}}. $$
According to \eqref{eq:gaussianLp}, we have
$$I\leq C_0(T)\left(  \int_{t_1}^{t_2} \frac{1}{(t-t_1)^{\frac{d}{p}\frac{q}{q-2}}} dt \right)^{\frac{q-2}{q}}. $$
For the last integral to be finite, one needs to have $\frac{d}{p}< \frac{q-2}{q}= 1-\frac{2}{q}$. This is exactly the constraint in Assumption~\ref{H1}.
Hence the desired result.

\end{proof}

\begin{lemma}
\label{Mainlemmas3}
Same assumptions as in Lemma \ref{MainLemmas1}. Let $C_0(T)$ be as in
Lemma \ref{MainLemmas1}. For any $\kappa>0$, there exists $C(T,\kappa)$ independent of $\mu_0$~
such that, for any $ 0\leq T_1\leq T_2 \leq T$ satisfying $  T_2-T_1 <
(C_0(T)\kappa)^{-\frac{1}{\frac{q-2}{q} - \frac{d}{p}}}$,
$$\forall x\in C([0,T];\R),~~\E^{\mathcal{G}_{T_1}}_\Ps \left[\exp \left\{\kappa \int_{T_1}^{T_2}
|b(t,w_t,x_t)|^2 dt\right\}\right]\leq C(T,\kappa).$$
\end{lemma}
\begin{proof}
Admit
for a while we have shown that there exists a constant $C(\kappa, T)$
such that for any $M \in \N$
\begin{equation}
\label{prel:star}
\sum_{k=1}^M \frac{\kappa^k}{k!} \E^{\mathcal{G}_{T_1}}_\Ps \left(
\int_{T_1}^{T_2} |b(t,w_t,x_t)|^2 \ dt\right)^{k}\leq C(T,\kappa),
\end{equation}
provided that $T_2-T_1 < (C_0(T)\kappa)^{-\frac{1}{\frac{q-2}{q} - \frac{d}{p}}}$. The desired result
then follows from Fatou's lemma.

We now prove \eqref{prel:star}.

By the tower property of conditional expectation,
\begin{align*}
& \E^{\mathcal{G}_{T_1}}_{\Ps}\left[\left(\int_{T_1}^{T_2}
|b(t,w_t,x_t)|^2dt\right)^{k}\right]= k!
\int_{T_1}^{T_2}\int_{t_1}^{T_2} \int_{t_2}^{T_2}\cdots \int_{t_{k-2}}^{T_2}  \int_{t_{k-1}}^{T_2} \E^{\mathcal{G}_{T_1}}_{\Ps} \Big[ |b(t_1,w_{t_1},x_{t_1})|^2
\\
&
\times |b(t_2,w_{t_2},x_{t_2})|^2 \dots\times |b(t_{k-1},w_{t_{k-1}},x_{t_{k-1}})|^2\left(
 \E^{\mathcal{G}_{t_{k-1}}}_{\Ps} |b(t_{k},w_{t_{k}},x_{t_{k}})|^2 \right)\Big]\,dt_k \,dt_{k-1}\cdots\,dt_2
\,dt_1 .
\end{align*}
In view of Lemma \ref{MainLemmas1},
$$\int_{t_{k-1}}^{T_2} \E^{\mathcal{G}_{t_{k-1}}}_{\Ps}
|b(t_{k},w_{t_{k}},x_{t_{k}})|^2\,dt_k \leq C_0(T) (T_2-t_{k-1})^{\frac{q-2}{q} - \frac{d}{p}}\leq C_0(T)
(T_2-T_1)^{\frac{q-2}{q} - \frac{d}{p}}.$$
Therefore, by Fubini's theorem,
\begin{align*}
&\E^{\mathcal{G}_{T_1}}_{\Ps}\left[\left(\int_{T_1}^{T_2}
|b(t,w_t,x_t)|^2 dt\right)^{k}\right]\leq k!  C_0(T)
(T_2-T_1)^{\frac{q-2}{q} - \frac{d}{p}}
\int_{T_1}^{T_2} \int_{t_1}^{T_2}\int_{t_2}^{T_2}\cdots  \int_{t_{k-2}}^{T_2}\\
&  \times \E^{\mathcal{G}_{T_1}}_{\Ps}\Big[
|b(t_{1},w_{t_{1}},x_{t_{1}})|^2 |b(t_{2},w_{t_{2}},x_{t_{2}})|^2 \times \dots\times |b(t_{k-1},w_{t_{k-1}},x_{t_{k-1}})|^2\Big]dt_{k-1}\cdots\,dt_2
\,dt_1.
\end{align*}
Now we repeatedly condition with respect to $\mathcal{G}_{t_{k-i}}$ ($i\geq 2$)~ and
combine Lemma \ref{MainLemmas1} with Fubini's theorem. It comes:
\begin{multline*}
\E^{\mathcal{G}_{T_1}}_{\Ps}\left(\int_{T_1}^{T_2}
|b(t,w_t,x_t)|^2 dt\right)^{k}
\leq k!  (C_0(T) (T_2-T_1)^{\frac{q-2}{q} - \frac{d}{p}})^{k-1} \\
\times \int_{T_1}^{T_2}\E^{\mathcal{G}_{T_1}}_{\Ps} |b(t_{1},w_{t_{1}},x_{t_{1}})|^2 dt_1
\leq k!  (C_0(T) (T_2-T_1)^{\frac{q-2}{q} - \frac{d}{p}} )^{k}.
\end{multline*}
Thus, \eqref{prel:star} is satisfied provided that $T_2-T_1 < (C_0(T)\kappa)^{-\frac{1}{\frac{q-2}{q} - \frac{d}{p}}}$.
\end{proof}

\begin{proof}[Proof of Proposition \ref{est1}]
Observe that
\begin{align}
\label{est1:1}
&\E_{\Ps} \exp\left\{ \alpha \int_0^T |b(t,w_t,Y_t)|^2dt\right\} =
\int_{C([0,T];\R^d)} \E_{\Ps}\exp\left\{ \alpha
\int_0^T|b(t,w_t,x_t)|^2dt\right\}  \Ps^{Y}(dx).
\end{align}
Set  $\delta: = \frac{1}{2C_0^2 T \alpha^2} \wedge T$, where $C_0$ is
as in Lemma~\ref{MainLemmas1}. Set $n:= \left[ \frac{T}{\delta}
\right]$.
  Then,
 \begin{align*}
 &\exp\left\{\alpha \int_{0}^{T}|b(t,w_t,x_t)|^2 dt\right\}
 ={\displaystyle\prod_{m=0}^{n}} \exp\left\{\alpha
\int_{(T-(m+1)\delta)\vee 0}^{T-m\delta}|b(t,w_t,x_t)|^2\,dt\right\},
\end{align*}
where $x$ is a fixed path.
Condition the right-hand side by $\mathcal{G}_{(T-\delta)\vee
0}$.
Notice that $\delta$ is small enough to be in the setting
of Lemma~\ref{Mainlemmas3}. Thus,
\begin{align*}
&\E_{\Ps} \exp\left\{ \alpha \int_0^T
|b(t,w_t,x_t)|^2dt\right\} \leq C(T,\alpha)\E_{\Ps}
{\displaystyle\prod_{m=1}^{n}} \exp\left\{\kappa N
\int_{(T-(m+1)\delta)\vee 0}^{T-m\delta}|b(t,w_t,x_t)|^2\,dt\right\}.
\end{align*}
Successively, conditioning by $\mathcal{G}_{(T-(m+1))\vee 0}$ for $
m=1, 2,
\dots n$ and using Lemma \ref{Mainlemmas3},
$$\E_{\Ps} \exp\left\{ \alpha \int_0^T
|b(t,w_t,x_t)|^2dt\right\} \leq C^n(T,\alpha)\E_{\Ps}
\exp\left\{\int_0^{(T-n\delta)\vee 0} b(t,w_t,x_t)dt\right\}
\leq  C(T,\alpha).$$
The proof is completed by plugging the preceding estimate
into~\eqref{est1:1}.
\end{proof}

\textbf{Sketch of the proof under Assumption~\ref{H2}.}
The main point is to adapt Lemma~\ref{MainLemmas1}. The rest is then straightforward. 
Starting as in the proof of Lemma~\ref{MainLemmas1}, we have, in view of Assumption~\ref{H2},
\begin{align*}
I&:=\int_{t_1}^{t_2}\E^{\mathcal{G}_{t_1}}_\Ps |b(t,w_t,x_t)|^2dt \leq \int_{t_1}^{t_2}\int h^2_t(y+w_{t_1}-x_t)g_{t-t_1}(y)dy \ dt\\
&= \int_{t_1}^{t_2}\left(\int_{B_{(0,1)}} h^2_t(y)g_{t-t_1}(y-w_{t_1}+x_t)dy  + \int_{B^c_{(0,1)}} h^2_t(y)g_{t-t_1}(y-w_{t_1}+x_t)dy\right)\ dt
\end{align*}
Applying twice H\"older's inequality as before and bounding any integral of $g$ that is not on the whole space with the one on the whole space, we obtain 
$$I\leq C \|h\|_{L^q((0,T);L^p_{loc}(\R^d))}\left(  \int_{t_1}^{t_2} \frac{1}{(t-t_1)^{\frac{d}{p}\frac{q}{q-2}}} dt \right)^{\frac{q-2}{q}}+ H(T) (t_2-t_1).$$
Which leads to the same conclusion as the one of Lemma~\ref{MainLemmas1} with the following constant 
$$C_0(T) = \|h\|_{L^q((0,T);L^p_{loc}(\R^d))} +  T^{\frac{2}{q} + \frac{d}{p}} .$$
\section{Propagation of chaos}
\label{sec:propchaos}
\subsection{Girsanov transform for $1 \leq r<N$ particles}
\label{sec:GirsTrK}
For any integer $1\leq
r<N$, proceeding as in the proof of Theorem~\ref{PSexistence}  one gets
the existence of a weak solution on $[0,T]$ to
\begin{equation}
\label{PS_Kdrift}
\begin{cases}
&d\widehat{X}_t^{l,N}=  dW_t^l, \quad 1\leq l\leq r,\\
& d \widehat{X}_t^{i,N}=  \left\{\frac{1}{N} \sum_{j=r+1}^N 
b(t,\widehat{X}_t^{i,N},\widehat{X}_t^{j,N})\right\}
dt + dW_t^i, \quad r+1\leq i\leq N,\\
& \widehat{X}_0^{i,N} 
 \text{  i.i.d. and independent of }(W):=(W^i, 1\leq i\leq N).
\end{cases}
\end{equation}
Below we set $\hat{X}:=(\hat{X}^{i,N}, 1\leq i\leq N)$ and we denote by
$\Q^{r,N}$ the probability measure under which $\hat{X}$ is well
defined.
Notice that
$(\widehat{X}^{l,N}, 1\leq l\leq r)$ is independent of
$(\widehat{X}^{i,N}, r+1\leq i\leq N)$.  We now study the exponential
local martingale associated to the change of drift between
\eqref{PS} and \eqref{PS_Kdrift}. For $x\in C([0,T];\R^d)^N$ set
\begin{equation*}
 \beta^{(r)}_t(x):= \Big(b_t^{1,N}(x), \dots, b_t^{r,N}(x), \frac{1}{N}
 \sum_{i=1}^r b(t,x_t^{r+1},x^{i}_t), \dots,
\frac{1}{N} \sum_{i=1}^r  b(t,x_t^{N},x^{i}_t)  \Big) .
\end{equation*}
In the sequel we will need uniform w.r.t~$N$ bounds for moments of
\begin{equation} \label{def:Z^(k)}
Z_T^{(r)}:= \exp\left\{ -\int_0^T \beta^{(r)}_t(\widehat{X}) \cdot
dW_t -\frac{1}{2}\int_0^T| \beta^{(r)}_t(\widehat{X})|^2 dt \right\}.
\end{equation}
\begin{proposition}
\label{GirsanovKparticles}
For any $T>0$, $\gamma>0$ and $r\geq 1$ there exists  $N_0 \geq r$ and
$ C(T, \gamma, r)$ s.t.
\begin{equation*}
\forall N\geq N_0, \quad  \E_{\Q^{r,N}} \exp\left\{ \gamma
\int_0^T|\beta^{(r)}_t(\widehat{X})|^2dt\right\}\leq C(T, \gamma, r).
\end{equation*}
\end{proposition}
\begin{proof}
For $x\in C([0,T];\R^d)^N$, one has
\begin{align*}
|\beta^{(r)}_t(x)|^2&= \sum_{i=1}^r \left( \frac{1}{N}\sum^N_{{j=1}}
b(t,x_t^i,x_t^j)\right)^2
+
\frac{1}{N^2}\sum_{j=1}^{N-r}\left(\sum_{i=1}^r 
b(t,x^{r+j}_t,x^{i}_t)\right)^2.
\end{align*}
By Jensen's inequality,
$$|\beta^{(r)}_t|^2 \leq \frac{1}{N}\sum_{i=1}^r \sum_{j=1}^N |b(t,x_t^i,x_t^j)|^2 + \frac{r}{N^2}\sum_{j=1}^{N-r}\sum_{i=1}^r |b(t,x^{r+j}_t,x^{i}_t)|^2. $$
For simplicity we below write $\E$ (respectively, $\hat{X}^i$) instead
of
$\E_{\Q^{r,N}}$ (respectively, $\hat{X}^{i,N}$).
Observe that
\begin{align*}
& \E\exp\Big\{ \gamma
\int_0^T|\beta^{(r)}_t(\widehat{X})|^2dt\Big\}\\
& \leq  \Big( \E \exp\Big\{\sum_{i=1}^r \frac{2\gamma
}{N}\sum_{j=1}^N\int_0^T |b(t,\widehat{X}_t^i,\widehat{X}_t^j)|^2
dt\Big\}\Big)^{1/2}  \Big( \E \exp\Big\{  \frac{2\gamma  r}{N^2}
\sum_{j=1}^{N-r}\sum_{i=1}^r
\int_0^T|b(t,\widehat{X}_t^{r+j},\widehat{X}^i_t)|^2  dt\Big\}\Big)^{1/2}\\
&\leq  \Big( \prod_{i=1}^r \frac{1}{N}
\sum_{j=1}^N\E\exp\Big\{2\gamma  r \int_0^T
|b(t,\widehat{X}_t^i,\widehat{X}_t^j)|^2  dt\Big\}  \Big)^{\frac{1}{2r}}
\Big( \prod_{j=1}^{N-r} \frac{1}{r}\sum_{i=1}^r
\E \exp \Big\{\frac{2\gamma  r^2}{N}
\int_0^T|b(t,\widehat{X}_t^{r+j},\widehat{X}_t^i)  dt \Big\}
\Big)^{\frac{1}{2(N-r)}}.
\end{align*}
In view of Proposition \ref{est1}, the proof is finished.
\end{proof}

\subsection{Tightness}
We start with showing the tightness of $\{\mu^{N}\}$ and of an auxiliary empirical measure which is needed in the sequel.
\begin{lemma}
\label{Tightness:crit}
Let $\Q^N$ be as above.
The sequence $\{\mu^{N}\}$ is tight under $\Q^N$. In addition, let
$\nu^N:= \frac{1}{N^4}\sum_{i,j,k,l=1}^N
\delta_{X^{i,N}_.,X^{j,N}_.,X^{k,N}_.,X^{l,N}_.}$. The sequence
$\{\nu^{N}\}$ is tight under $\Q^N$.
\end{lemma}
\begin{proof}
The tightness of $\{\mu^{N}\}$,
respectively $\{\nu^{N}\}$, results from the tightness of the
intensity
measure $\{\E_{\Q^N} \mu^N(\cdot)\}$, respectively$\{\E_{\Q^N}
\nu^N(\cdot)\}$: See Sznitman~\cite[Prop. 2.2-ii]{Sznitman}. 
By symmetry, in both cases it suffices to check the
tightness of $\{\text{Law}(X^{1,N})\}$. We aim to prove
\begin{equation}
\label{PS:KolmCrit}
\exists C>0, \forall N\geq N_0, \quad
\E_{\Q^N}[|X_t^{1,N}-X_s^{1,N}|^4] \leq C_T |t-s|^2, \quad 0\leq
s,t\leq T,
\end{equation}
where $N_0$ is as in Proposition \ref{GirsanovKparticles}. Let
$Z_T^{(1)}$ be as in~(\ref{def:Z^(k)}). One has
$$\E_{\Q^N}[|X_t^{1,N}-X_s^{1,N}|^4] =
\E_{\Q^{1,N}}[(Z_T^{(1)})^{-1}|\widehat{X}_t^{1,N}-\widehat{X}_s^{1,N}|^4].
$$
As $\widehat{X}^{1,N}$ is a Brownian motion under
$\Q^{1,N}$,
$$\E_{\Q^N}[|X_t^{1,N}-X_s^{1,N}|^4]\leq
(\E_{\Q^{1,N}}[(Z_T^{(1)})^{-2}])^{1/2}
(\E_{\Q^{1,N}}[|\widehat{X}_t^{1,N}-\widehat{X}_s^{1,N}|^8])^{1/2}\leq
(\E_{\Q^{1,N}}[(Z_T^{(1)})^{-2}])^{1/2}  C |t-s|^2.$$
Observe that, for a Brownian motion $(W^\sharp)$ under $\Q^{1,N}$,
$$\E_{\Q^{1,N}}[(Z_T^{(1)})^{-2}]=\E_{\Q^{1,N}} \exp\left\{ 2
\int_0^T\beta_t^{(1)}(\hat{X})\cdot dW_t^\sharp - \int_0^T
|\beta^{(1)}_t(\hat{X})|^2dt\right\} .$$
Adding and subtracting  $3\int_0^T|\beta^{(1)}_t|^2dt$ and applying
again the Cauchy-Schwarz inequality,
$$\E_{\Q^{1,N}}[(Z_T^{(1)})^{-2}]\leq \left(\E_{\Q^{1,N}}
\exp\left\{6\int_0^T|\beta_t^{(1)}(\hat{X})|^2dt\right\}\right)^{1/2}.$$
Applying Proposition \ref{GirsanovKparticles} with $k=1$ and
$\gamma=6$, we obtain the desired result.
\end{proof}
\subsection{Convergence}
To prove Theorem~\ref{th:PSconvergence} we have to show that any limit
point of $\{\text{Law}(\mu^N)\}$ is $\delta_\Q$, where $\Q$ is the
unique solution to (MP).

Let $\phi \in C_b(\R^{ad})$, $f\in C_b^2(\R^d)$, $0<t_1< \cdots <t_a\leq s
<t \leq T$ and $m\in \Pp(C[0,T];\R^d)$. Set
\begin{multline*}
G(m):= \int_{(C[0,T];\R^d)^2} \phi(x^1_{t_1}, \dots, x^1_{t_a})
\Big(f(x^1_t) -f(x^1_s) \\
-\frac{1}{2}\int_s^t \triangle f(x^1_u)du - \int_s^t
\nabla f(x^1_u) \cdot b(u,x^1_u,x^2_u) du\Big)dm(x^1)\otimes dm(x^2).
\end{multline*}
We start with showing that
\begin{equation}
\label{PoC:1}
\lim_{N \to \infty}\E [\left(G(\mu^N)\right)^2]=0.
\end{equation}
Observe that
\begin{align*}
& G(\mu^N)= \frac{1}{N}\sum_{i=1}^N \phi(X^{i,N}_{t_1}, \dots, X^{i,N}_{t_a})  \Big(f(X_t^{i,N})-f(X_s^{i,N})- \frac{1}{2}\int_s^t\triangle f(X^{i,N}_{u})du   \\
&  -\frac{1}{N} \sum_{j=1}^N \int_s^t \nabla f(X^{i,N}_u) \cdot
b(u,X^{i,N}_u,X^{j,N}_u) \ du \Big).
\end{align*}
Apply It\^o's formula to $\frac{1}{N}\sum_{i=1}^N(f(X_t^{i,N})-f(X_s^{i,N}))$, it is easy to verify that $\E [\left(G(\mu^N)\right)^{2}]\leq\frac{C}{N}$.
Thus, \eqref{PoC:1} holds true.

Suppose for a while we have proven the following lemma:
\begin{lemma}
\label{PoC:biglemma}
Let $\Pi^\infty \in \mathcal{P}(\mathcal{P}(C([0,T];\R^d)^4))$ be a limit point of $\{\text{law}(\nu^N)\}$.
Then
\begin{equation}
\label{PoC:BigLimit}
\begin{split}
& \lim_{N \to \infty}\E [\left(G(\mu^N)\right)^2]=
\int_{\mathcal{P}(C([0,T];\R^d)^4)} \left\{ \int_{C([0,T];\R^d)^4} \Big[
f(x^1_t)
-f(x^1_s)-\frac{1}{2}\int_s^t \triangle f(x^1_u)du \right. \\
&\left. ~~ - \int_s^t \nabla f (x^1_u)\cdot b(u,x^1_u,x^2_u) du\Big]
\times \Big[
f(x^3_t)
-f(x^3_s)-\frac{1}{2}\int_s^t \triangle f(x^3_u)du \right. \\
&\left. ~~ - \int_s^t \nabla f (x^3_u)\cdot b(u,x^3_u,x^4_u) du\Big] \times \phi(x^1_{t_1}, \dots, x^1_{t_a}) \phi(x^3_{t_1}, \dots, x^3_{t_a})d\nu(x^1,\ldots,x^4)
\vphantom{\frac15}\right\} d\Pi^\infty(\nu),
\end{split}
\end{equation}
and
\begin{enumerate}[i)]
\item Any $\nu \in \mathcal{P}(C([0,T];\R^d)^4)$ belonging to the support of $\Pi^\infty $ is a product measure: $\nu = \nu^1 \otimes \nu^1 \otimes \nu^1
\otimes \nu^1$.
\item For any $t\in (0,T]$, the time marginal $\nu^1_t$ of $\nu^1$ has
a density $\rho_t^1$ which satisfies for any $r>1$
$$\exists C_T, ~ \forall 0<t\leq T, ~ ~ \|\rho_t^1\|_{L^r(\R^d)}\leq \frac{C_T}{t^{\frac{d}{2}(1-\frac{1}{r})}}.$$
\end{enumerate}
\end{lemma}
Then, combining \eqref{PoC:1} with the above result, we get
\begin{align*}
&\int_{C([0,T];\R^d)}  \phi(x^1_{t_1}, \dots, x^1_{t_a})\Big [ f(x_t^1)-f(x_s^1)- \frac{1}{2}\int_s^t \triangle f(x_u)du   -  \int_s^t \nabla f(x^1_{u}) \cdot b(u,x^1_{u}-y)\rho_u^1(y)dy  du\Big]d\nu^1(x^1)=0.
\end{align*}
We deduce that $\nu^1$ solves (MP) and thus that $\nu^1=\Q$. As
by definition $\Pi^\infty$ is a limit point of $\text{Law}(\nu^N)$, it
follows that any limit point of $\text{Law}(\mu^N)$ is $\delta_\Q$,
which ends the proof.
\subsubsection{Proof of Lemma \ref{PoC:biglemma}}
\paragraph*{Proof of~\eqref{PoC:BigLimit}: Step 1.}
Notice that
\begin{align}
\label{PoC:2}
& \E [\left(G(\mu^N)\right)^{2}] = \frac{1}{N^2}\E \sum_{i,k=1}^N \Phi_2(X^{i,N},X^{k,N})
+  \frac{1}{N^3}\E\sum_{i,k,l=1}^N \Phi_3(X^{i,N},X^{k,N},X^{l,N}) \nonumber \\
& + \frac{1}{N^3}\E\sum_{i,j,k=1}^N \Phi_3(X^{k,N},X^{i,N},X^{j,N}) +
\frac{1}{N^4} \E\sum_{i,j,k,l=1}^N \Phi_4(X^{i,N},X^{j,N},X^{k,N},
X^{l,N}),
\end{align}
where
\begin{align*}
& \Phi_2(X^{i,N},X^{k,N}) := \phi(X^{i,N}_{t_1}, \dots,
X^{i,N}_{t_a})~\phi(X^{k,N}_{t_1}, \dots, X^{k,N}_{t_a})\\
& \times  \Big(f(X_t^{i,N})-f(X_s^{i,N})-
\frac{1}{2}\int_s^t \triangle f(X^{i,N}_{u})du\Big)\Big(f(X_t^{k,N})-f(X_s^{k,N})-
 \frac{1}{2}\int_s^t \triangle f(X^{k,N}_{u})du\Big), \\
&  \Phi_3(X^{i,N},X^{k,N},X^{l,N}) := -\phi(X^{i,N}_{t_1}, \dots,
X^{i,N}_{t_a})~\phi(X^{k,N}_{t_1}, \dots, X^{k,N}_{t_a})\\
& \times \Big(f(X_t^{i,N})-f(X_s^{i,N})-
\frac{1}{2}\int_s^t\triangle f(X^{i,N}_{u_1})du_1\Big)\int_s^t
\nabla f (X^{k,N}_u)\cdot b(u,X^{k,N}_u,X^{l,N}_u)  \mathbb{1}_{{\{(X_u^{k,N},X_u^{l,N})\notin \mathcal{N}_b(u)\}}}\ du, \\
& \Phi_4(X^{i,N},X^{j,N},X^{k,N}, X^{l,N}) :=
\phi(X^{i,N}_{t_1}, \dots,X^{i,N}_{t_a})
\phi(X^{k,N}_{t_1},\dots,X^{k,N}_{t_a})
\\ & \times \int_s^t \int_s^t  \nabla f(X^{i,N}_{u_1})\cdot b(u_1,X^{i,N}_{u_1},X^{j,N}_{u_1}) \mathbb{1}_{{\{(X_{u_1}^{i,N},X_{u_1}^{j,N})\notin \mathcal{N}_b(u_1)\}}} 
\\
& ~~~~~~~~~~~~~~~~~~~~~\times \nabla f (X^{k,N}_{u_2}) \cdot b(u_2,X^{k,N}_{u_2},X^{l,N}_{u_2}) \mathbb{1}_{{\{(X_{u_2}^{i,N},X_{u_2}^{j,N})\notin \mathcal{N}_b(u_2)\}}} \ d{u_1} \ d{u_2}.
\end{align*}
Let $C_N$ be the last term in the r.h.s. of~\eqref{PoC:2}.
In Steps 2-4 below we prove that $C_N$ converges as $N \to \infty$
and we
identify its limit.~
Define the function~$F$ on $\R^{(2p+4)d}$ as
\begin{multline} \label{def:F}
F(x^1, \dots, x^{2p+4}) := \phi(x^{5},\dots,x^{p+4})~\phi(x^{p+5}, 
\dots, x^{2p+4})~\nabla f(x^1)\cdot b(u_1,x^1,x^2) 
\nabla f(x^3)\cdot b(u_2,x^3,x^4)\\
\times \mathbb{1}_{{\{(x^{1},x^{2})\notin \mathcal{N}_b(u_1)\}}}  \mathbb{1}_{{\{(x^{3},x^{4})\notin \mathcal{N}_b(u_2)\}}}  .
\end{multline}
 We set $C_N=\int_s^t \int_s^t A_N \ du_1 \ du_2$ with
\begin{equation*}
A_N:=\frac{1}{N^4}\sum_{i,j,k,l=1}^N \E(F(X^{i,N}_{u_1},
X^{j,N}_{u_1},
X^{k,N}_{u_2},X^{l,N}_{u_2},X^{i,N}_{t_1},\dots,
 X^{i,N}_{t_a},X^{k,N}_{t_1}, \dots, X^{k,N}_{t_a})).
\end{equation*}
We now~aim to show that $A_N$ converges
pointwise (Step~2),~that~$|A_N|$ is bounded from above by an
integrable function w.r.t.~$d{\theta_1}\ d{\theta_2}\ d{u_1} \ d{u_2}$
(Step~3), and finally to identify the limit of $C_N$ (Step~4).
\paragraph*{Proof of~\eqref{PoC:BigLimit}: Step 2.}
Fix $u_1,u_2 \in [s,t]$ .  Define $\tau^N$ as
$$\tau^N:= \frac{1}{N^4}\sum_{i,j,k,l=1}^N \delta_{X^{i,N}_{u_1},
X^{j,N}_{u_1},
X^{k,N}_{u_2},X^{l,N}_{u_2},X^{i,N}_{t_1},\dots,
 X^{i,N}_{t_a},X^{k,N}_{t_1},\dots, X^{k,N}_{t_a}}.$$
Define the measure $\Q^N_{u_1,\theta_1,u_2,\theta_2,t_1, \dots,t_a}$
on  $(\R^{2a+4})^d$ as
$ \Q^N_{u_1,u_2,t_1, \dots,t_a}(A)= \E(\tau^N(A))$.
The convergence of $\{\text{law}(\nu^N)\}$ implies the weak convergence
of $\Q^N_{u_1,\theta_1,u_2,\theta_2,t_1, \dots,t_a}$ to the measure on
$(\R^{2a+4})^d$ defined by
\begin{align*}
&\Q_{u_1,u_2,t_1,
\dots,t_a}(A):=\int_{\mathcal{P}(C([0,T];\R^d)^4)} \int_{C([0,T];\R^d)^4}
\mathbb{1}_{A}(x^1_{u_1}, x^2_{u_1},
x^3_{u_2},x^4_{u_2}, x^1_{t_1},\dots,\\
&~~~~~~~~~~~~~~~~~~~~~~~~~~~~~~~~~~~~~~~~~~~~~~~~~~~~~~~~~~~~~~~~~~~
x^1_{t_a},x^3_{t_1},\dots, x^3_{t_a})d\nu(x^1,x^2,x^3,x^4)d\Pi^{\infty}(\nu).
\end{align*}
Let us show that this probability measure has an $L^2$-density w.r.t. the Lebesgue measure~on~$(\R^{2a+4})^{\times d}$ ( $L^2$ could be replaced with any  $L^r$).
Let $h\in C_c(\R^{(2a+4)d})$. By weak convergence,
\begin{align*}
& \left| <\Q_{u_1,u_2,t_1, \dots,t_a},h> \right|\\
 &= \left| \lim_{N \to \infty} \frac{1}{N^4} \sum_{i,j,k,l=1}^N \E
 h(X^{i,N}_{u_1}, X^{j,N}_{u_1},
 X^{k,N}_{u_2},X^{l,N}_{u_2},
 X^{i,N}_{t_1},\dots, X^{i,N}_{t_a},X^{k,N}_{t_1},\dots,
 X^{k,N}_{t_a})\right|.
\end{align*}
When, in the preceding sum, at least two indices are equal, we bound the expectation by $\|h\|_\infty.$
When $i\neq j \neq k\neq l$, we apply Girsanov's transform in
Section \ref{sec:GirsTrK} with four particles
and Proposition~
\ref{GirsanovKparticles}. This procedure leads to
\begin{multline*}
\left| <\Q_{u_1,\theta_1,u_2,\theta_2,t_1, \dots,t_a},h> \right|\leq
\lim_{N \to \infty} \Big(\|h\|_\infty \frac{C}{N} \\
+  \frac{C_T}{N^4} \sum_{i\neq j \neq k \neq l} \left(\E
h^2(\hat{X}^{i,N}_{u_1}, \hat{X}^{j,N}_{u_1},
\hat{X}^{k,N}_{u_2},\hat{X}^{l,N}_{u_2}
,\hat{X}^{i,N}_{t_1},\dots, \hat{X}^{i,N}_{t_a},\hat{X}^{k,N}_{t_1},\dots,
\hat{X}^{k,N}_{t_a})\right)^{1/2}\Big).
\end{multline*}
All the processes $\hat{X}^{i,N},\ldots,\hat{X}^{l,N}$ being
independent Brownian motions we deduce that
$$ \left| <\Q_{u_1,u_2,t_1, \dots,t_a},h> \right|\leq
C_{u_1,u_2,\theta_1,\theta_2, t_1, \dots, t_a}
\|h\|_{L^2(\R^{2p+6})}.
$$
It follows from Riesz's representation~ theorem that $\Q_{u_1,u_2,t_1, \dots,t_a}$ has a density w.r.t. Lebesgue's measure in $L^2(\R^{(2a+4)d})$. Therefore,
the functional~$F$ is continuous $\Q_{u_1,u_2,t_1,
\dots,t_a}$ - a.e. Since for any fixed $u_1,u_2 \in
[s,t]$  $F$ is also bounded  $\Q_{u_1,u_2,t_1,
\dots,t_a}$ - a.e. we have
\begin{align*}
\lim_{N\to \infty} A_N = <\Q_{u_1,u_2,t_1,
\dots,t_a}, F>.
\end{align*}
\paragraph*{Proof of~\eqref{PoC:BigLimit}: Step 3.}
In view of the definition~(\ref{def:F}) of $F$ we may restrict 
ourselves to the case $i\neq j$ and
$k\neq l$. Use the Girsanov transforms from Section
\ref{sec:GirsTrK} with $r_{i,j,k,l}\in \{2,3,4\}$ according to 
the 
respective 
cases $(i=k,j=l)$, $(i=k,j\neq l)$, $(i\neq k,j\neq l)$, etc. 
Below we write $r$ instead of $r_{i,j,k,l}$.
By exchangeability it comes:
$$A_N=\Big| \frac{1}{N^4}\sum_{i\neq j,k\neq l}
\E_{\Q^{r,N}}(Z^{(r)}_TF(\cdots))\Big|\leq
\frac{1}{N^4} \sum_{i\neq j,k\neq l}
\left(\E_{\Q^{r,N}}(Z^{(r)}_T)^2\right)^{1/2}
\Big(\E_{\Q^{r,N}}(F^2(\cdots))\Big)^{1/2}.$$
By Proposition \ref{GirsanovKparticles}, $\E_{\Q^{r,N}}(Z^{(r)}_T)^2$
can be bounded uniformly w.r.t. $N$. As the functions $f$ and $\phi$
are bounded we deduce
$$\sqrt{\E_{\Q^{r,N}}(F^2(\cdots))} \leq C 
\left(\E_{\Q^{r,N}}(h^2_{u_1}(W^i_{u_1}-W^j_{u_1})
h^2_{u_2}(W^k_{u_2}-W^l_{u_2}))
 \right)^{1/2}, $$
for $i\neq j$, $k\neq l$ and $r\equiv r_{i,j,k,l}$. We consider the three cases:
\textbf{Case 1}  $i\neq k$, $j\neq l$ : As all $4$ Brownian motions are independent, one can separate this into a product of expectations and using the same computations as in Lemma \ref{MainLemmas1}, one has 
$$\left(\E_{\Q^{r,N}}(h^2_{u_1}(W^i_{u_1}-W^j_{u_1})
h^2_{u_2}(W^k_{u_2}-W^l_{u_2}))
 \right)^{1/2} \leq \sqrt{\|h_{u_1}\|^2_{L^p(\R^d)} \|g_{u_1}\|_{L^{\frac{p}{p-2}}(\R^d)}\|h_{u_2}\|^2_{L^p(\R^d)} \|g_{u_2}\|_{L^{\frac{p}{p-2}}(\R^d)}}$$
 
\textbf{Case 2} $i=k$, $j=l$ : As we ony have two independent Brownian motions, we condition by the smaller time index and by one of the two independent Brownian motions. It comes
\begin{multline*}
\left(\E_{\Q^{r,N}}(h^2_{u_1}(W^i_{u_1}-W^j_{u_1})
h^2_{u_2}(W^i_{u_2}-W^j_{u_2}))
 \right)^{1/2} \leq \mathbb{1}_{\{u_1< u_2\}}\Big(\|h_{u_1}\|^2_{L^p(\R^d)} \|g_{u_1}\|_{L^{\frac{p}{p-2}}(\R^d)}\\
 \times \|h_{u_2}\|^2_{L^p(\R^d)} \|g_{u_2-u_1}\|_{L^{\frac{p}{p-2}}(\R^d)}\Big)^{\frac{1}{2}}+ \mathbb{1}_{\{u_2< u_1\}}\Big(\|h_{u_1}\|^2_{L^p(\R^d)} \|g_{u_1-u_2}\|_{L^{\frac{p}{p-2}}(\R^d)}
 \times \|h_{u_2}\|^2_{L^p(\R^d)} \|g_{u_2}\|_{L^{\frac{p}{p-2}}(\R^d)}\Big)^{\frac{1}{2}}
\end{multline*}
\textbf{Case 3} $i=k$, $j\neq l$: Same bound as above is obtained by  conditioning by the smaller time index and $W^j$ and $W^l$.
In any of the above cases, in view of the assumption {\bf (H$^b$)}, the bounds are integrable in $L^1((0,T)^2)$.
We thus have obtained:
$A_N \leq C H(u_1, u_2),$
where $H$ belongs to $L^1((0,T)^2)$.

\paragraph*{Proof of~\eqref{PoC:BigLimit}: Step 4.} Steps 2 and 3
allow us to conclude that
\begin{align*}
 \lim_{N\to \infty} C_N = \int_s^t\int_s^t
<\Q_{u_1,u_2,t_1, \dots,t_a},F>
~du_1du_2.
\end{align*}
By definition of $\Q_{u_1,u_2,t_1, \dots,t_a}$ and
$F$ we thus have obtained that
\begin{align*}
\lim_{N\rightarrow \infty} C_N  =&\int_{P(C([0,T];\R)^4)}\int_s^t\int_s^t
\int_{C([0,T];\R^d)^4}   \phi(x^1_{t_1},
\dots, x^1_{t_a})\phi(x^3_{t_1}, \dots, x^3_{t_a})  \\
&\times  \nabla f(x^1_{u_1}) \cdot b(u_1,x^1_{u_1},x^2_{u_1})
 \nabla f(x^3_{u_2})\cdot b(u_2,x^3_{u_2},x^4_{u_2})\mathbb{1}_{{\{(x^1_{u_1},x_{u_1}^{2})\notin \mathcal{N}_b(u_1)\}}}  \mathbb{1}_{{\{(x_{u_2}^{3},x_{u_2}^{4})\notin \mathcal{N}_b(u_2)\}}}  \\
&~~~~~~~~d\nu(x^1,x^2,x^3,x^4)
~d\theta_1~d\theta_2~du_1~du_2~d\Pi^\infty(\nu).
\end{align*}
A similar procedure is applied to the three other terms in
the r.h.s. of~\eqref{PoC:2}. Together with the preceding,
we obtain~\eqref{PoC:BigLimit}
\paragraph*{Proof of~ i) and ii).}
Now, we prove the claims i) and ii) of Lemma~\ref{PoC:biglemma}.
\begin{enumerate}[i)]
\item For any measure $\nu \in \mathcal{P}(C([0,T];\R)^4)$, denote its first
marginal
by $\nu^1$. One easily gets
$ \Pi^\infty~~\text{a.e.},~~ \nu = \nu^1 \otimes \nu^1 \otimes \nu^1
\otimes \nu^1$
(see~\cite[Lemma 3.3]{BossyTalay}).
\item Take $\varphi \in C_c(\R^d)$ and fix $r>1$. Let $\alpha \in (1,r')$ where $r'$ is the conjugate of $r$. Using similar arguments as in the above
Step~1, for any $0<t\leq T$ one has $\Pi^\infty(d \nu)$ a.e.,
\begin{align*}
&<\nu^1_t, \varphi>= \lim_{N \to \infty}\E_{\Q^N} <\mu^N_t, h>= \lim_{N \to \infty}\E_{\Q^N}(\varphi(X_t^{1,N}))=\lim_{N \to \infty}\E_{\Q^{1,N}}(Z_T^{(1)}\varphi(W_t^{1,N}))\\
& \leq C \left(\E_{\Q^{1,N}}(Z_T^{(1)})^{\alpha'} \right)^\frac{1}{\alpha'} \left(\E_{\Q^{1,N}}(\varphi(X_t^{1,N}))^{\alpha} \right)^\frac{1}{\alpha}\leq C \|\varphi\|_{L^{r'}(\R^d)} \|g_t\|^{\frac{1}{\alpha}}_{L^{\left(r'/\alpha\right)'}}\leq C  \|\varphi\|_{L^{r'}(\R^d)}  \frac{1}{t^{\frac{d}{2}\frac{1}{r'}}}
\end{align*}
Thus, one has
$$<\nu^1_t, \varphi>\leq C  \|\varphi\|_{L^{r'}(\R^d)}  \frac{1}{t^{\frac{d}{2}(1-\frac{1}{r})}}. $$
Apply the Riesz representation theorem to conclude the proof.
\end{enumerate}

\bibliography{biblio}

\end{document}